\documentclass[reqno,12pt]{amsart}
\usepackage{amsmath,amsthm,amssymb,amsfonts,amscd}
\usepackage{epsfig}
\usepackage[all]{xy}
\numberwithin{equation}{section}

\theoremstyle{plain}
\newtheorem{theorem}{Theorem}

\newtheorem{corollary}[theorem]{Corollary}
\newtheorem{proposition}[theorem]{Proposition}
\newtheorem*{theorem*}{Theorem}
\newtheorem*{conjecture*}{Conjecture}

\theoremstyle{definition}
\newtheorem{remark}[theorem]{Remark}

\newtheorem*{definition}{Definition}

\newcommand{\CC}{{\mathbb{C}}}
\newcommand{\KK}{{\mathbb{K}}}
\newcommand{\HH}{{\mathbb{H}}}
\newcommand{\PP}{{\mathbb{P}}}
\newcommand{\QQ}{{\mathbb{Q}}}

\newcommand{\RR}{{\mathbb{R}}}
\newcommand{\ZZ}{{\mathbb{Z}}}

\newcommand{\E}[2]{{E_{#1}^{#2}}}
\newcommand{\F}[2]{{F_{#1}^{#2}}}

\def\E{{\mathcal E}}
\def\F{{\mathcal F}}

\def\O{{\mathcal O}}

\def\T{{\mathcal T}}

\def\p{\partial }
\def\ns{{\nabla}\hspace{-1.4mm}\raisebox{0.3mm}{\text{\footnotesize{\bf /}}}}

\begin{document}
\title{On rational Frobenius manifolds of rank three with symmetries}
\date{\today}
\author{Alexey Basalaev}
\address{National Research University Higher School of Economics, Vavilova 7, 117312 Moscow, Russia}
\email{aabasalaev@edu.hse.ru}
\address{Leibniz Universit\"at Hannover, Welfengarten 1, 30167 Hannover, Germany}
\email{basalaev@math.uni-hannover.de}
\author{Atsushi Takahashi}
\address{Department of Mathematics, Graduate School of Science, Osaka University, 
Toyonaka Osaka, 560-0043, Japan}
\email{takahashi@math.sci.osaka-u.ac.jp}
\begin{abstract}
  We study Frobenius manifolds of rank three and dimension one that are related to submanifolds of certain Frobenius manifolds arising in mirror symmetry of elliptic orbifolds. 
  We classify such Frobenius manifolds that are defined over an arbitrary field $\KK\subset \CC$ via the theory of modular forms. 
By an arithmetic property of an elliptic curve $\E_\tau$ defined over $\KK$ associated to such a Frobenius manifold, it is proved that there are only two such Frobenius manifolds defined over $\CC$ satisfying a certain symmetry assumption and thirteen Frobenius manifolds defined over $\QQ$ satisfying a weak symmetry assumption on the potential.
\end{abstract}
\maketitle
\section*{Introduction}
The notion of a Frobenius manifold was introduced by Boris Dubrovin in the 90s (cf.\ \cite{du:1}) as the mathematical axiomatization of a 2D topological conformal field theory. A special class of Frobenius manifolds is given by certain structures on the base space of the universal unfolding of a hypersurface singularity. These structures were introduced in the early 80s by Kyoji Saito  (cf.\ \cite{st:1} for an introduction to this theory) and called at that time Saito's flat structures. 

Actually, Saito found a richer structure than his flat structure, consisting of the filtered de Rham cohomology with the Gau\ss--Manin connection, higher residue pairings and a primitive form \cite{sa:1}. Unlike the general setting of a Frobenius manifold it has much more geometric data coming naturally from singularity theory. 
It is also now generalized as a so-called non-commutative Hodge theory by \cite{kkp} which will be a necessary tool to understand 
the classical mirror symmetry (isomorphism of Frobenius manifolds) via a Kontsevich's homological mirror symmetry.
It is a very important problem to study some arithmetic aspect of a Saito structure with a geometric origin such as singularity theory. 
However, it is quite difficult at this moment. Therefore we start our consideration from the larger setting of Frobenius manifolds.

After recalling some basic definitions and terminologies in Section~\ref{sec:preliminary}, we shall study a rational structure on the class of Frobenius manifolds of rank three and dimension one,
that is natural from the point of view of the theory of primitive forms. This class is obtained by applying the ${\rm GL}(2, \CC)$-action to the special solution of the WDVV equation that has a geometric origin. We study it in Section~\ref{sec:rank3FM}.
In particular, the ${\rm GL}(2, \CC)$-action can be thought of as a transformation changing one primitive form to another one.
More precisely, we shall define the Frobenius manifold $M^{(\tau_0,\omega_0)}$ of rank three and dimension one obtained by acting with
a certain element $A^{(\tau_0,\omega_0)} \in {\rm GL}(2,\CC)$ depending on $\tau_0 \in \HH$, $\omega_0 \in \CC \backslash \{0\}$ 
(see Subsection~\ref{choice})
on  the ``basic'' Frobenius manifold $M^\infty$ (see Proposition~\ref{specialsolution} for the definition of $M^\infty$). 
Let $E_k(\tau)$ be the $k$-th Eisenstein series, $E_2^*(\tau) := E_2(\tau) - \frac{3}{\pi {\rm Im}(\tau)}$ and $\E_{\tau_0}$ the elliptic curve with the modulus $\tau_0$. Then the first theorem of this paper is the following.
  
  \begin{theorem*}[Theorem~\ref{theorem1} in Section~\ref{sec:FMoverK}]
    Let $\KK\subset \CC$ be a field. Let $\tau_0\in\HH$ and $\omega_0\in\CC\backslash\{0\}$. 
    The following are equivalent$:$
    \begin{enumerate}
      \item
	The Frobenius manifold $M^{(\tau_0,\omega_0)}$ is defined over $\KK$.
      \item
	All the coefficients of $f^{(\tau_0,\omega_0)}(t)$ series expansion are in $\KK$.
      \item 
	We have
	\begin{equation*} 
	E_2^*(\tau_0)\in \KK\omega_0^2,
	\quad E_4(\tau_0)\in\KK\omega_0^4,
	\quad E_6(\tau_0) \in\KK\omega_0^6.
	\end{equation*}
      \item
	Let $\partial$ be the almost holomorphic derivative defined by \eqref{def of almost holomorphic derivative}. We have
	\begin{equation*} 
	-\frac{1}{24}E_2^*(\tau_0)\in \KK\omega_0^2,
	\quad -\frac{1}{24}\p E_2^*(\tau_0)\in\KK\omega_0^4,
	\quad -\frac{1}{24}\p^2E_2^*(\tau_0) \in\KK\omega_0^6.
	\end{equation*}
      \item
	We have
	\begin{equation*}
	E_2^*(\tau_0)\in \KK\omega_0^2, \quad \E_{\tau_0} \ \text{is defined over} \ \KK.
	\end{equation*}
    \end{enumerate}
  \end{theorem*}
We shall also give two natural examples over $\QQ$ in Subsection~\ref{examples}. 
  In what follows we translate some properties of the elliptic curve $\E_{\tau_0}$ into special properties of the Frobenius manifold $M^{(\tau_0, \omega_0)}$. Considering the ${\rm SL}(2,\RR)$-action on $M^{(\tau_0,\omega_0)}$ we define the property of the Frobenius manifold $M^{(\tau_0,\omega_0)}$ to be ``symmetric'' and ``weakly symmetric'' (Definition~\ref{defn:symmetry}). We relate these properties to special properties of the elliptic curve $\E_{\tau_0}$ in the next theorem.
  \begin{theorem*}[Theorem~\ref{theorem2} in Section~\ref{sec:Symmetries}]
    Let $\tau_0 \in \HH$ and $\omega_0 \in \CC \backslash \{0\}$.
    \begin{enumerate}
      \item The Frobenius manifold $M^{(\tau_0, \omega_0)}$ has a symmetry if and only if $\tau_0$ is in the $\rm SL(2, \ZZ)$ orbit of $\sqrt{-1}$ or $\rho$.
      \item The Frobenius manifold $M^{(\tau_0, \omega_0)}$ defined over $\QQ$ has a weak symmetry if and only if $\tau_0$ is from the list given in Corollary~\ref{cor:13EllipticModulus}.
    \end{enumerate}
  \end{theorem*}
Finally, some useful data are given in the Appendix.
  
\begin{sloppypar}

{\bf Acknowledgements}.\  
We are grateful to Wolfgang Ebeling for fruitful discussions.
The second named author is supported by JSPS KAKENHI Grant Number 24684005.
\end{sloppypar}
\section{Preliminaries}\label{sec:preliminary}
\subsection{Frobenius manifolds}
We give some basic properties of a Frobenius manifold \cite{du:1}. Let us recall the equivalent definition taken from Saito-Takahashi \cite{st:1}.
\begin{definition}
Let $M=(M,\O_{M})$ be a connected complex manifold of dimension $\mu$
whose holomorphic tangent sheaf and cotangent sheaf 
are denoted by $\T_{M}$ and $\Omega_M^1$ respectively
and let $d$ be a complex number.
A {\it Frobenius structure of rank $\mu$ and dimension $d$ on M} is a tuple $(\eta, \circ , e,E)$, where $\eta$ is a non-degenerate $\O_{M}$-symmetric bilinear form on $\T_{M}$, $\circ $ is an $\O_{M}$-bilinear product on $\T_{M}$, defining an associative and commutative $\O_{M}$-algebra structure with a unit $e$, and $E$ is a holomorphic vector field on $M$, called the Euler vector field, which are subject to the following axioms:
\begin{enumerate}
\item The product $\circ$ is self-adjoint with respect to $\eta$: that is,
\begin{equation*}
\eta(\delta\circ\delta',\delta'')=\eta(\delta,\delta'\circ\delta''),\quad
\delta,\delta',\delta''\in\T_M. 
\end{equation*} 
\item The {\rm Levi}--{\rm Civita} connection $\ns:\T_M\otimes_{\O_M}\T_M\to\T_M$ 
with respect to $\eta$ is flat: that is, 
\begin{equation*}
[\ns_\delta,\ns_{\delta'}]=\ns_{[\delta,\delta']},\quad \delta,\delta'\in\T_M.
\end{equation*}
\item The tensor $C:\T_M\otimes_{\O_M}\T_M\to \T_M$  defined by 
$C_\delta\delta':=\delta\circ\delta'$, $(\delta,\delta'\in\T_M)$ is flat: that is,
\begin{equation*}
\ns C=0.
\end{equation*} 
\item The unit element $e$ of the $\circ $-algebra is a 
$\ns$-flat holomorphic vector field: that is,
\begin{equation*}
\ns e=0.
\end{equation*} 
\item The metric $\eta$ and the product $\circ$ are homogeneous of degree 
$2-d$ ($d\in\CC$) and $1$ respectively with respect to the Lie derivative 
$Lie_{E}$ of the {\rm Euler} vector field $E$: that is,
\begin{equation*}
Lie_E(\eta)=(2-d)\eta,\quad Lie_E(\circ)=\circ.
\end{equation*}
\end{enumerate}
A manifold $M$ equipped with a Frobenius structure $(\eta, \circ , e,E)$ is called a {\it Frobenius manifold}.
\end{definition}
From now on in this section, we shall always denote by $M$ a Frobenius manifold.
We expose some basic properties of Frobenius manifolds without their proofs.
Let us consider the space of horizontal sections of the connection $\ns$:
\[
\T_M^f:=\{\delta\in\T_M~|~\ns_{\delta'}\delta=0\text{ for all }\delta'\in\T_M\}
\]
which is a local system of rank $\mu $ on $M$ such that the metric $\eta$ 
takes a constant value on $\T_M^f$. Namely, we have 
\begin{equation*}
\eta (\delta,\delta')\in\CC,\quad  \delta,\delta' \in \T_M^f.
\end{equation*}
\begin{proposition}\label{prop:flat coordinates}
At each point of $M$, there exist local coordinates $(t_1,\dots,t_{\mu})$,  called flat coordinates, such that $e=\p_1$, $\T_M^f$ is spanned by $\p_1,\dots, \p_{\mu}$ and $\eta(\p_i,\p_j)\in\CC$ for all $i,j=1,\dots, \mu$, where we denote $\p/\p t_i$ by $\p_i$. 
\end{proposition} 
The axiom $\ns C=0$ implies the following:
\begin{proposition}\label{prop:potential}
At each point of $M$, there exists a local holomorphic function $\F$, called Frobenius potential, satisfying
\begin{equation*}
\eta(\p_i\circ\p_j,\p_k)=\eta(\p_i,\p_j\circ\p_k)=\p_i\p_j\p_k \F,
\quad i,j,k=1,\dots,\mu,
\end{equation*}
for any system of flat coordinates. In particular, one has
\begin{equation*}
\eta_{ij}:=\eta(\p_i,\p_j)=\p_1\p_i\p_j \F. 
\end{equation*}
\end{proposition}
The product structure on $\T_{M}$ is 
described locally by ${\mathcal F}$ as 
\begin{equation*} \label{eq:prod}
{\partial_{i}} \circ 
{\partial_{j}}  = \sum_{k = 1}^{\mu}
c_{ij}^{k} {\partial_{k}}  
\quad i,j = 1, \cdots, \mu,
\end{equation*}
\begin{equation*} \label{eq:str-const}
c_{ij}^{k} :=\sum_{l=1}^{\mu} \eta^{kl} 
\partial_{i} \partial_{j} \partial_{l} \F, \quad 
(\eta^{ij}) = (\eta_{ij})^{-1}, \quad
i,j,k = 1, \cdots, \mu.
\end{equation*}
\begin{definition}
Let $\KK\subset\CC $ be a field. 
We say that a Frobenius manifold $M$ is {\it defined over $\KK$} 
if there exist flat coordinates $t_1,\dots, t_{\mu}$ such that the Frobenius potential $\F$ belongs to $\KK\{t_1,\dots, t_{\mu}\}$ and is defined at the point $t_1 = \dots = t_\mu = 0$.
\end{definition}
\subsection{Eisenstein series}
Throughout this paper, we denote by $\HH$ the complex upper half plane $\{\tau\in\CC~\vert~{\rm Im}(\tau)>0\}$.
Recall the following famous facts on Eisenstein series.
\begin{proposition}
Let $E_2(\tau)$, $E_4(\tau)$ and $E_6(\tau)$ be the Eisenstein series defined by 
\begin{subequations}
\begin{equation*}
E_2(\tau) := 1 - 24 \sum_{n=1}^\infty \sigma_1(n) q^n,
\end{equation*}
\begin{equation*}
E_4(\tau) := 1 + 240 \sum_{n=1}^\infty \sigma_3(n) q^n,
\end{equation*}
\begin{equation*}
E_6(\tau) := 1 - 504 \sum_{n=1}^\infty \sigma_5(n) q^n,
\end{equation*}
\end{subequations}
where $\sigma_k(n)=\displaystyle\sum_{d\vert n}d^k$ and $q = \exp(2 \pi \sqrt{-1} \tau)$.
\begin{enumerate}
\item
For any 
$
\begin{pmatrix}
a &b\\
c & d
\end{pmatrix}
\in {\rm SL}(2,\ZZ)
$, we have 
\begin{subequations}
\begin{equation}\label{eq:modularE2}
E_2(\tau)=\frac{1}{(c\tau+d)^2}E_2\left(\frac{a\tau+b}{c\tau+d}\right)-\frac{6c}{\pi\sqrt{-1}(c\tau+d)},
\end{equation}
\begin{equation}
E_4(\tau)=\frac{1}{(c\tau+d)^4}E_4\left(\frac{a\tau+b}{c\tau+d}\right),
\end{equation}
\begin{equation}
E_6(\tau)=\frac{1}{(c\tau+d)^6}E_6\left(\frac{a\tau+b}{c\tau+d}\right).
\end{equation}
\end{subequations}
\item
The derivatives of the Eisenstein series satisfy the following identities due to Ramanujan$:$
\begin{subequations}\label{eq:E_kDerivatives}
\begin{equation}
\frac{1}{2\pi \sqrt{-1}} \dfrac{d E_2(\tau)}{d\tau} = \frac{1}{12} \left( E_2(\tau)^2 - E_4(\tau) \right),
\end{equation}
\begin{equation}
\frac{1}{2\pi \sqrt{-1}} \dfrac{d E_4(\tau)}{d\tau} = \frac{1}{3} \left( E_2(\tau) E_4(\tau) - E_6(\tau) \right),
\end{equation}
\begin{equation}
\frac{1}{2\pi \sqrt{-1}} \dfrac{d E_6(\tau)}{d\tau}  = \frac{1}{2} \left( E_2(\tau) E_6(\tau) - E_4(\tau)^2 \right).
\end{equation}
\end{subequations}
\end{enumerate}
\end{proposition}
We shall also consider the complex-valued real-analytic function $E_2^*(\tau)$ on $\HH$ defined by  
\begin{equation*}
E_2^*(\tau) := E_2(\tau) - \frac{3 }{\pi {\rm Im} (\tau)},
\end{equation*}
which is a so-called almost holomorphic modular form of weight two since we have the following. 
\begin{proposition}
We have 
\begin{equation*}
E_2^*(\tau)=\frac{1}{(c\tau+d)^2}E_2^*\left(\frac{a\tau+b}{c\tau+d}\right) \text{ for any } 
\begin{pmatrix}
a &b\\
c & d
\end{pmatrix}
\in {\rm SL}(2,\ZZ).
\end{equation*}
\end{proposition}
\begin{proof}
The formula
\begin{equation}\label{eq:19}
\left({\rm Im}\left(\frac{a\tau+b}{c\tau+d}\right)\right)^{-1}= \frac{|c\tau+d|^2}{{\rm Im (\tau)}},\quad  
\begin{pmatrix}
a &b\\
c & d
\end{pmatrix}
\in {\rm SL} (2,\ZZ),
\end{equation}
yields the statement.
\end{proof}
In general, an almost holomorphic modular form is defined as follows.
\begin{definition}
A polynomial $f(\tau)$ in ${\rm Im} (\tau)^{-1}$ over the ring of holomorphic functions on $\HH$ satisfying 
\begin{equation*}
f(\tau)=\frac{1}{(c\tau+d)^k}f\left(\frac{a\tau+b}{c\tau+d}\right) \text{ for any } 
\begin{pmatrix}
a &b\\
c & d
\end{pmatrix}
\in {\rm SL} (2,\ZZ),
\end{equation*}
is called an {\em almost holomorphic modular form} of weight $k$.
\end{definition}
\begin{proposition}[cf.\ Paragraph~5.1. in \cite{Z}]
Let $f(\tau)$ be an almost holomorphic modular form of weight $k$. 
Then the almost holomorphic derivative of $f(\tau)$ defined by
\begin{equation}\label{def of almost holomorphic derivative}
\p_k f (\tau) := \frac{1}{2 \pi \sqrt{-1}} \frac{\p f(\tau)}{\p\tau} - \frac{k}{4\pi {\rm Im} (\tau)} f(\tau),
\end{equation}
is an almost holomorphic modular form of weight $k+2$.
\end{proposition}
\begin{proof}
One can check this directly by using the equations \eqref{eq:modularE2} and \eqref{eq:19}.
 We briefly explain for the reader's convenience the modularity property of $\p_2 E_2^* (\tau)$. We have:
  \begin{equation*}
    \p_2 E_2^* = \frac{1}{12} \left( E_2(\tau)^2 - E_4(\tau) \right) - \frac{3}{4 \pi ({\rm Im}(\tau))^2} - \frac{1}{2 \pi {\rm Im}(\tau)} E_2^*(\tau).
  \end{equation*}
  \begin{equation*}
    = \frac{1}{12} \left( E_2(\tau)^2 - \frac{6 E_2(\tau)}{\pi {\rm Im}(\tau)} + \frac{9}{(\pi {\rm Im}(\tau))^2}\right) - \frac{1}{12} E_4(\tau) .
  \end{equation*}
  \begin{equation*}
    = \frac{1}{12} E_2^*(\tau)^2 - \frac{1}{12} E_4(\tau).
  \end{equation*}
  Due to the modularity properties of $E_4$ and $E_2^*$ the proposition follows.
\end{proof}

In what follows we will drop the subscript $k$ in the derivative keeping in mind that it is always fixed as we are given a modular form of weight $k$ to differentiate. We will use the notation $\partial ^p$ meaning:
$$
  \partial^p g := \partial_{k+2(p-1)} \dots \partial_k g,
$$
for $g$ - an almost holomorphic modular form of weight $k$.
\begin{proposition}
We have
\begin{subequations}
\begin{equation*}
\p E_2^*(\tau) = \frac{1}{12} \left( E_2^*(\tau)^2 - E_4(\tau)\right),
\end{equation*}
\begin{equation*}
\partial^2 E_2^*(\tau) = \frac{1}{36} \left( E_6(\tau)  - \frac{3}{2} E_2^*(\tau) E_4(\tau) + \frac{1}{2} E_2^*(\tau)^2\right).
\end{equation*}
\end{subequations}
\end{proposition}
\begin{proof}
This follows from direct calculations using the equations \eqref{eq:E_kDerivatives}.
\end{proof}

\subsection{Elliptic curves}\label{sec:elliptic curve}
We have a family of elliptic curves parameterized by $\HH:$
\begin{equation*}
\pi:\E:=\{(x,y,\tau)\in\CC^2\times \HH~\vert~y^2=4x^3-g_2(\tau)x-g_3(\tau)\}\longrightarrow \HH,
\end{equation*}
where
\begin{equation}\label{eq:g to Eisentein}
g_2(\tau):= \frac{4\pi^4}{3} E_4(\tau), \quad g_3(\tau):=\frac{8 \pi ^6}{27} E_6(\tau).
\end{equation}
Denote by $\E_{\tau_0}$ the fiber of $\pi$ over a point $\tau_0\in\HH$.
\begin{definition}
Let $\KK\subset \CC$ be a field. 
Choose a point $\tau_0\in\HH$.
We say that an elliptic curve $\E_{\tau_0}$ is {\it defined over $\KK$} 
if there exist $g_2,g_3\in\KK$ such that the algebraic variety 
\begin{equation*}
E_{g_2,g_3}:=\{(x,y)\in\CC^2~\vert~y^2 = 4 x^3 -g_2 x - g_3\}
\end{equation*}
is isomorphic to $\E_{\tau_0}$.
\end{definition}
\section{Frobenius manifolds of rank three and dimension one}\label{sec:rank3FM}
From now on, we shall consider a Frobenius manifold $M$ of rank three and dimension one with 
flat coordinates $t_1,t_2,t$ satisfying the following conditions$:$
\begin{itemize}
\item The unit vector field $e$ is given by $\frac{\p}{\p t_1}$.
\item The Euler vector field $E$ is given by  
$E=t_1\frac{\p}{\p t_1}+\frac{1}{2}t_2\frac{\p}{\p t_2}$.
\item The Frobenius potential $\F$ is given by
\begin{equation}\label{eq:potential}
\F=\frac{1}{2}t_1^2 t+t_1 t_2^2+t_2^4 f(t)
\end{equation}
where $f(t)$ is a holomorphic function in $t$ on an open domain in $\CC$.
\end{itemize}
\subsection{Solutions of the WDVV equation}
\begin{proposition}
The WDVV equation is equivalent to the following differential equation.
\begin{equation}\label{eq:WDVV}
\frac{d^3f(t)}{dt^3}=-24 f(t) \frac{d^2f(t)}{dt^2}+36\left(\frac{df(t)}{dt}\right)^2.
\end{equation}
\end{proposition}
\begin{proof}
This is obtained by a straightforward calculation.
\end{proof}
\begin{remark}
Put 
\begin{equation*}
\gamma(t):=-4f(t).
\end{equation*}
By a straightforward calculation, 
it turns out that the holomorphic function $\gamma(t)$ satisfies the following differential equation
\begin{equation}\label{eq:Chazy}
\frac{d^3\gamma(t)}{dt^3}=6\frac{d^2\gamma(t)}{dt^2}\gamma(t)-9\left(\frac{d\gamma(t)}{dt}\right)^2.
\end{equation}
The differential equation \eqref{eq:Chazy} is classically known as Chazy's equation.
\end{remark}
\begin{proposition}
Suppose that $f(t)$ is a convergent power series in $t$ given as ${f(t)=\displaystyle\sum_{n=0}^\infty \frac{c_n}{n!}t^n}$, then the differential equation \eqref{eq:WDVV} is equivalent to 
the following recursion relation$:$
\begin{equation}\label{eq:recursion}
c_{n+3}= \sum_{a = 0}^n {n \choose a} \left( -24  c_a c_{n -a + 2} + 36c_{a+1} c_{n-a+1} \right).
\end{equation}
In particular, we have
\begin{equation*}
c_{3}=-24c_2c_0 + 36c_1^2.
\end{equation*}
\end{proposition}
\begin{proof}
This is also obtained by a straightforward calculation.
\end{proof}
Therefore, the first three coefficients $c_0$, $c_1$ and $c_2$ are enough to determine all the coefficients $c_n, n \ge 3$ due to the recursion relation \eqref{eq:recursion}.
\subsection{${\rm GL}(2,\CC)$-action on the set of Frobenius structures}
\begin{proposition}\label{prop:GL on X}
Suppose that a holomorphic function $f(t)$ on a domain in $\CC$ is a 
solution of the differential equation \eqref{eq:WDVV}.
For any 
$
A=
\begin{pmatrix}
a &b\\
c & d
\end{pmatrix}
\in {\rm GL}(2,\CC)
$, define a holomorphic function $f^A(t)$ on a suitable domain in $\CC$ as 
\begin{equation}\label{eq:GL-action}
f^A(t):=\frac{\det(A)}{(ct+d)^2}f\left(\frac{at+b}{ct+d}\right)+\frac{c}{2(ct+d)}.
\end{equation}
Then $f^A(t)$ becomes a solution of the differential equation \eqref{eq:WDVV}.
\end{proposition}
\begin{proof}
This is obtained by a straightforward calculation.
\end{proof}
It is important to note that this ${\rm GL}(2,\CC)$-action is the inverse action of the ${\rm GL}(2,\CC)$-action on the set of solutions of the WDVV equations for the potential \eqref{eq:potential} given in Appendix B in \cite{du:1}. Indeed, we have the following.
\begin{proposition}\label{prop:GL on Frob}
Consider Dubrovin's inversion $I$ of the Frobenius manifold defined as follows$:$ 
\begin{equation*}
\widehat{t}_1:=t_1 + \frac{1}{4}\frac{t_2^2}{t},\ 
\widehat{t}_2:=\frac{t_2}{t},\ \widehat{t}:=-\frac{1}{t},
\end{equation*}
\begin{equation*}
\widehat{\F}(\widehat{t}):=\frac{1}{t^2}\left[\F(t)- t_1^2 t - \frac{1}{4}t_1 t_2^2 \right].
\end{equation*}
Then, the new Frobenius manifold given by the new flat coordinates $\widehat{t}_1, \widehat{t}_2, \widehat{t}$ 
together with the new Frobenius potential $\widehat{\F}$ 
coincides with the one associated to the solution $f^A(t)$ of~\eqref{eq:WDVV} with 
$A=\begin{pmatrix}
0 &-1\\
1 & 0
\end{pmatrix}$.
\end{proposition}
\begin{proof}
Some calculations yield the statement. 
\end{proof}
\subsection{${\rm GL}(2,\CC)$-orbit of constant solutions}
The differential equation \eqref{eq:WDVV} obviously has constant solutions. 
Therefore, we have the following.
\begin{proposition}
For any $e\in \CC$ and any point $[c:d]\in\PP^1$, the meromorphic function on~$\CC$
\begin{equation}\label{eq:210}
f(t) := \frac{e}{(ct+d)^2}+\frac{c}{2(ct +d)}
\end{equation}
is a solution of the differential equation \eqref{eq:WDVV}.
\end{proposition}
\begin{proof}
This is clear.
\end{proof}
\begin{definition}
We will call the solution $f(t)$ as above constant solution of the equation~\eqref{eq:WDVV}.
\end{definition}
\begin{corollary}
If $f(t)$ is holomorphic at $t=0$ and belongs to the ${\rm GL}(2,\CC)$-orbit of constant solutions, then 
there exist $\alpha,\beta\in\CC$ such that 
\begin{equation}\label{eq:trivial}
f(t) := \frac{\alpha}{(1+\beta t)^2}+\frac{\beta}{2(1+\beta t)}.
\end{equation}
\end{corollary}
\begin{proof}
One can set $\alpha:=e/d^2$ and $\beta:=c/d$ in \eqref{eq:210} since $d$ cannot be zero.
\end{proof}
The Taylor expansion at $t=0$ of $f(t)$ in the equation \eqref{eq:trivial} is given by 
\begin{subequations}
\begin{equation*}
f(t)=c_0(\alpha,\beta)+c_1(\alpha,\beta)t+c_2(\alpha,\beta)\frac{t^2}{2} +\dots, 
\end{equation*}
\begin{equation*}
c_0(\alpha,\beta)=\alpha+\frac{\beta}{2},\quad c_1(\alpha,\beta)=-2\alpha\beta-\frac{1}{2}\beta^2,\quad c_2(\alpha,\beta)=6\alpha\beta^2+\beta^3.
\end{equation*}
\end{subequations}
For some $c_0,c_1,c_2\in\CC$, consider the cubic curve in $\CC^2$ defined by 
\begin{equation*}
y^2=4x^3-12c_0 x^2 -6c_1x-\frac{c_2}{2}.
\end{equation*}
Note that if $c_i=c_i(\alpha,\beta)$ for $i=1,2,3$,
then the cubic curve is singular since we have
\begin{equation*}
4x^3-12c_0(\alpha,\beta) x^2 -6c_1(\alpha,\beta)x-\frac{c_2(\alpha,\beta)}{2}
=\frac{1}{2}(2x-6\alpha-\beta)(2x-\beta)^2.
\end{equation*}
\begin{proposition}
Suppose that $c_0,c_1,c_2\in\CC$ satisfy the equation 
$$32(c_1+2c_0^2)^3-(c_2 + 12c_1c_0 + 16c_0^3)^2=0.$$
Then there exist $\alpha(c_0,c_1,c_2),\beta(c_0,c_1,c_2)\in\CC$ such that 
\begin{equation*}
\begin{aligned}
c_0=&\alpha(c_0,c_1,c_2)+\frac{\beta(c_0,c_1,c_2)}{2},\\
c_1=&-2\alpha(c_0,c_1,c_2)\beta(c_0,c_1,c_2)-\frac{1}{2}\beta(c_0,c_1,c_2)^2,\\
c_2=&6\alpha(c_0,c_1,c_2)\beta(c_0,c_1,c_2)^2+\beta(c_0,c_1,c_2)^3,
\end{aligned}
\end{equation*}
and the unique solution $f(t)$ of the differential equation \eqref{eq:WDVV} holomorphic at $t=0$ satisfying
\begin{equation*}
f(0) =c_0,\quad \frac{df}{dt}(0)=c_1,\quad \frac{d^2f}{dt^2}(0)=c_2
\end{equation*}
is given by
\begin{equation*}
f(t) := \frac{\alpha(c_0,c_1,c_2)}{(1+\beta(c_0,c_1,c_2) t)^2}+\frac{\beta(c_0,c_1,c_2)}{2(1+\beta(c_0,c_1,c_2) t)}.
\end{equation*}
\end{proposition}
\begin{proof}
  Consider the system of PDE's called Halphen's system of equations:
  \begin{equation*}\label{eq:Halphen}
    \begin{cases}
      &\frac{d}{dt} (X_2(t) + X_3(t)) = 2X_2(t)X_3(t), \\
      &\frac{d}{dt} (X_3(t) + X_4(t)) = 2X_3(t)X_4(t),\\
      &\frac{d}{dt} (X_4(t) + X_2(t)) = 2X_4(t)X_2(t),
    \end{cases}
  \end{equation*}
  One can check that the function defined by $f(t) := -\frac{1}{6} (X_2(t) + X_3(t)+ X_4(t))$ is a solution of the equation \eqref{eq:WDVV}. 
  Consider the third order equation in $x$:
  $$
    4x^3 - 12 f(t) x^2 - 6f^\prime(t) x - \frac{f^{\prime \prime}(t)}{2}= 0,
  $$
  where $t$ is considered as a parameter. Let $\{x_k(t)\}$ be the triplet of solutions of a this third order equation. By straightforward computations one checks that the unordered triplet $\{-2X_k(t)\}$ is equal to the triplet $\{x_k(t)\}$.
  
  The discriminant $\Delta^Q$ of the third order equation at $t=0$ is equal to $32(c_1+2c_0^2)^3-(c_2 + 12c_1c_0 + 16c_0^3)^2=0$. In this case it is easy to solve Halphen's system to get:
  $$
    X_3(t) = X_4(t) = -\frac{\beta}{1+ \beta t}, \quad X_2(t) = -\frac{\beta}{1+\beta t} -6 \frac{\alpha}{(1+\beta t)^2}.
  $$
  Hence the function $f(t)$ is of the right form and solves equation \eqref{eq:WDVV}.
\end{proof}
\subsection{Special solution}
Under the change of variables $t = 2\pi \sqrt{-1} \tau$ the equation \eqref{eq:WDVV} transforms to:
\begin{equation}\label{eq:WDVV2}
\frac{d^3f(\tau)}{d\tau^3}=-48\pi\sqrt{-1}\frac{d^2f(\tau)}{d\tau^2}f(\tau)+72\pi\sqrt{-1}\left(\frac{df(\tau)}{d\tau}\right)^2.
\end{equation}
\begin{proposition}\label{specialsolution}
The holomorphic function $f^\infty(\tau)$ defined on $\HH$ by:
\begin{equation}\label{def of f infinity}
f^\infty(\tau):=-\frac{1}{24}E_2\left(\tau\right)
\end{equation}
satisfies the differential equation \eqref{eq:WDVV2}. 
Therefore, the holomorphic function $\F^\infty$ on $M^\infty:=\CC^2\times \HH$ given by 
\begin{equation*}
\F^\infty=\frac{1}{2}t_1^2(2\pi\sqrt{-1}\tau)+t_1 t_2^2+t_2^4 f^\infty (\tau)
\end{equation*}
defines on $M^\infty$ a Frobenius structure of rank three and dimension one. 
\end{proposition}
\begin{proof}
This follows from a direct calculation by the use of the equations \eqref{eq:E_kDerivatives}.
\end{proof}
\begin{remark}
It is a well-known consequence of the equations \eqref{eq:E_kDerivatives} that 
the function $\frac{\pi \sqrt{-1}}{3} E_2(\tau)$ satisfies the Chazy equation \eqref{eq:Chazy} with $t=\tau$ 
(cf.\ Appendix~C in \cite{du:1}).
\end{remark}
\begin{proposition}
The holomorphic function $2\pi\sqrt{-1}f^{\infty}(\tau)$ is invariant under the ${\rm SL}(2,\ZZ)$-action \eqref{eq:GL-action}.
\end{proposition}
\begin{proof}
This follows from a direct calculation using the modular property \eqref{eq:modularE2} of $E_2(\tau)$.
\end{proof}
\subsection{Choice of a primitive form and a ${\rm GL}(2,\CC)$-action}\label{choice}
We may consider the Frobenius manifold $M^\infty$ as a Frobenius submanifold of the one constructed from the invariant theory of the elliptic Weyl group of type $D_4^{(1,1)}$ depending on the particular choice of a vector in a two dimensional vector space (this is a direct consequence of Theorem ~2.7 in \cite{st:2}). 
A Frobenius structure varies according to a choice of this vector which is identified with a cycle in the homology group of the elliptic curve (see also Example~2 in Section 3.3 in \cite{sa:1}). 
We introduce a special class of ${\rm GL}(2,\CC)$-actions on the function $2\pi \sqrt{-1} f^\infty(\tau)$ motivated by this.
Consider a free abelian group $H_\ZZ$ generated by two letters $\alpha,\beta$
\[
H_{\ZZ} := \ZZ \alpha \oplus \ZZ \beta
\]
equipped with a symplectic form $(-,-)$ such that $(\alpha,\beta) = 1$. 
We can identify $H_\ZZ$ with the homology group $H_1(\E_{\sqrt{-1}},\ZZ)$ of an elliptic curve $\E_{\sqrt{-1}}$, 
the fiber at $\sqrt{-1}\in\HH$ of the family of elliptic curves $\pi:\E\longrightarrow \HH$ (see Subsection~\ref{sec:elliptic curve}).
Then 
\[
H^*_\CC := (H_\CC)^*:=(H_\ZZ \otimes_\ZZ \CC)^* = \CC \alpha^\vee \oplus \CC \beta^\vee,
\]
where $\{\alpha^\vee, \beta^\vee\}$ is the dual basis of $\{\alpha,\beta\}$, can be identified with the cohomology group 
$H^1(\E_{\sqrt{-1}},\ZZ)$.
In particular, the relative holomorphic volume form $\Omega\in \Gamma(\HH,\Omega^1_{\E/\HH})$ is described in terms of $\alpha^\vee, \beta^\vee$ as 
\begin{equation*}
\Omega=x(\tau)\left(\alpha^\vee+\tau \beta^\vee\right)
\end{equation*}
for some nowhere vanishing holomorphic function $x(\tau)$ on $\HH$. 
The relative holomorphic volume form $\zeta^\infty=\alpha^\vee+\tau\beta^\vee$ is, very roughly speaking, the {\it primitive form} associated to the choice of the vector $\alpha\in H_\CC$, which satisfies 
\begin{equation*}
\int_{\alpha}\zeta^\infty=1\text{ and }\int_\beta\zeta^\infty=\tau,
\end{equation*}
and gives the Frobenius structure $M^\infty$.
There is a systematic way to obtain a primitive form by the use of the canonical opposite filtration to the Hodge filtration corresponding to a point $\tau_0\in \HH$ as follows.
\begin{proposition}
For $\tau_0\in\HH$ and $\omega_0\in\CC\backslash\{0\}$, 
there exists a unique relative holomorphic volume form $\zeta\in \Gamma(\HH,\Omega^1_{\E/\HH})$ such that 
\begin{equation*}
\int_{\alpha'}\zeta=1,\quad \alpha':=\frac{1}{\omega_0( \bar \tau_0 - \tau_0)}\left(\bar{\tau}_0\alpha-\beta\right).
\end{equation*}
\end{proposition}
\begin{proof}
Some calculation yields
\[
\zeta =\omega_0\frac{ \bar \tau_0 - \tau_0}{\bar \tau_0 - \tau}\left(\alpha^\vee+\tau \beta^\vee\right).
\]
\end{proof}
This holomorphic volume form $\zeta$ is the primitive form uniquely determined by the choice of 
the vector $\alpha'\in H_\CC$. We first fixe $\tau_0\in\HH$ and $\omega_0\in\CC\backslash\{0\}$ so that we have
\begin{equation*}
\int_\alpha\zeta =\omega_0\text{ and }\int_\beta\zeta=\omega_0\tau_0\quad \text{at\ }\tau=\tau_0. 
\end{equation*}
Next we choose $\beta'\in H_\CC$ so that $\int_{\beta'}\zeta=0$ at $\tau=\tau_0$ and $(\alpha',\beta')=1$. It is easy to see that 
\[
\beta':=-\omega_0\left({\tau}_0\alpha-\beta\right).
\]
As the flat coordinate $2\pi\sqrt{-1}\tau$ of the Frobenius manifold $M^\infty$ associated to the primitive form $\zeta^\infty$, define the coordinate $t(\tau)$ by the period 
\begin{equation*}
\frac{t(\tau)}{2\pi \sqrt{-1}} :=\int_{\beta'}\zeta=  2\sqrt{-1}\omega_0^2{\rm Im}(\tau_0) \frac{\tau_0 - \tau}{\bar \tau_0 - \tau}.
\end{equation*}
This motivates the following ${\rm GL}(2,\CC)$-action $A^{(\tau_0,\omega_0)}$ and the Frobenius manifold $M^{(\tau_0,\omega_0)}$.
\begin{definition}
Choose $\tau_0\in\HH$ and $\omega_0\in\CC\backslash\{0\}$. 
\begin{enumerate}
\item
Define a holomorphic function $f^{(\tau_0,\omega_0)}(t)$ on $\{t\in\CC~|~|t|<|-4\pi\omega_0^2{\rm Im}(\tau_0)|\}$ applying the ${\rm GL}(2,\CC)$-action \eqref{eq:GL-action} specified by
\begin{equation*}
A^{(\tau_0,\omega_0)} :=
\begin{pmatrix}
\dfrac{\bar{\tau}_0}{4\pi\omega_0{\rm Im}(\tau_0)} & \omega_0 \tau_0\\
\dfrac{1}{4\pi\omega_0{\rm Im}(\tau_0)} & \omega_0
\end{pmatrix}
\end{equation*}
to the function $2\pi \sqrt{-1} f^\infty(\tau)$.
\item
Define complex numbers $c_i(\tau_0,\omega_0)$, $i\in\ZZ_{\ge 0}$, by the coefficients of the Taylor expansion of $f^{(\tau_0,\omega_0)}(t)$ at $t=0:$
\begin{equation*}
f^{(\tau_0,\omega_0)}(t)=\sum_{n=0}^\infty \frac{c_n(\tau_0,\omega_0)}{n!}t^n.
\end{equation*} 
\item 
Denote by $M^{(\tau_0,\omega_0)}:=\CC^2\times \{t\in\CC~|~|t|<|-4\pi\omega_0^2{\rm Im}(\tau_0)|\}$ the Frobenius manifold given by 
the Frobenius potential 
\begin{equation*}
\F^{(\tau_0,\omega_0)}=\frac{1}{2}t_1^2 t+t_1 t_2^2+t_2^4 f^{(\tau_0,\omega_0)}(t).
\end{equation*}
\end{enumerate}
\end{definition}
\section{Frobenius manifolds $M^{(\tau_0,\omega_0)}$ defined over $\KK$ via modular forms}\label{sec:FMoverK}

The essential technique dealing with the ${\rm GL}(2,\CC)$-action is provided by the theory of modular forms. We use it to give a complete classification of the Frobenius manifolds $M^{(\tau_0,\omega_0)}$ defined over $\KK \subset \CC$.

\subsection{Classification of $M^{(\tau_0,\omega_0)}$ defined over $\KK$}
\begin{theorem}\label{theorem1}
Let $\KK\subset \CC$ be a field. Let $\tau_0\in\HH$ and $\omega_0\in\CC\backslash\{0\}$. 
The following are equivalent$:$
\begin{enumerate}
\item
The Frobenius manifold $M^{(\tau_0,\omega_0)}$ is defined over $\KK$.
\item
All the coefficients of $f^{(\tau_0,\omega_0)}(t)$ series expansion are in $\KK$.
\item 
We have
\begin{equation*} 
E_2^*(\tau_0)\in \KK\omega_0^2,
\quad E_4(\tau_0)\in\KK\omega_0^4,
\quad E_6(\tau_0) \in\KK\omega_0^6.
\end{equation*}
\item
Let $\partial$ be the almost holomorphic derivative defined by \eqref{def of almost holomorphic derivative}. We have
\begin{equation*} 
-\frac{1}{24}E_2^*(\tau_0)\in \KK\omega_0^2,
\quad -\frac{1}{24}\p E_2^*(\tau_0)\in\KK\omega_0^4,
\quad -\frac{1}{24}\p^2E_2^*(\tau_0) \in\KK\omega_0^6.
\end{equation*}
\item
We have
\begin{equation*}
E_2^*(\tau_0)\in \KK\omega_0^2, \quad \E_{\tau_0} \ \text{is defined over} \ \KK.
\end{equation*}

\end{enumerate}
\end{theorem}
\begin{proof}
By definition, the Frobenius manifold $M^{(\tau_0,\omega_0)}$ is defined over $\KK$ if and only if 
there are flat coordinates $t_1,\widetilde{t}_2,\widetilde{t}$ such that the Frobenius potential is given by
\begin{equation*}
\F^{(\tau_0,\omega_0)}=\frac{1}{2}\eta_1{t}_1^2 \widetilde{t}+\eta_2{t}_1 \widetilde{t}_2^2+\widetilde{t}_2^4 \widetilde{f}(\widetilde{t})\quad \text{for some }\eta_1,\eta_2\in\KK
\text{ and }\widetilde{f}(\widetilde{t})\in\KK\{t\}.
\end{equation*}
However, this immediately implies that $t_2^2=\eta_2\widetilde{t}_2^2$, $t=\eta_1\widetilde{t}$ and $f^{(\tau_0,\omega_0)}(t)=\eta_2^{-2}\widetilde{f}(\widetilde{t})$, and hence the equivalence between the conditions (i) and (ii).

Due to the recursion relation \eqref{eq:recursion} to get (iii) it is enough to check that $c_i(\tau_0,\omega_0) \in \KK$ for $ 2 \ge i \ge 0$.
By definition of $f^\infty(\tau)$ (see \eqref{def of f infinity}), we have
\[
f^{(\tau_0,\omega_0)}(t) = -\frac{(4\pi\omega_0^2{\rm Im}(\tau_0))^2}{24\omega_0^2 (t+4\pi\omega_0^2{\rm Im}(\tau_0))^2} E_2\left(\frac{\bar{\tau_0}t +4\pi\omega_0^2{\rm Im}(\tau_0)\tau_0}{t+4\pi\omega_0^2{\rm Im}(\tau_0)} \right) 
+ \frac{1}{2(t+4\pi\omega_0^2{\rm Im}(\tau_0))}.
\]
Setting $t = 0$, we get
\[
c_0(\tau_0,\omega_0) = -\frac{1}{24\omega_0^2} E_2(\tau_0)+ \frac{1}{8\pi\omega_0^2{\rm Im}(\tau_0)} 
= -\frac{1}{24\omega_0^2} \left( E_2(\tau_0)-\frac{3}{\pi {\rm Im}(\tau_0)}\right).
\]
Using the formula \eqref{eq:E_kDerivatives}, we compute the derivative of $f^{(\tau_0,\omega_0)}(t)$ at $t=0$ and we obtain
\begin{align*}
c_1(\tau_0,\omega_0) =& \frac{1}{12\omega_0^2(4\pi\omega_0^2{\rm Im}(\tau_0))} E_2(\tau_0) - \frac{1}{288\omega_0^4} 
\left( E_2(\tau_0)^2 - E_4(\tau_0)\right)  - \frac{1}{2(4\pi\omega_0^2{\rm Im}(\tau_0))^2}\\
=&-2\left(-\frac{1}{24\omega_0^2} E_2(\tau_0)+ \frac{1}{8\pi\omega_0^2{\rm Im}(\tau_0)}\right)^2
+\frac{1}{288\omega_0^4}E_4(\tau_0)\\
=& -2c_0(\tau_0,\omega_0)^2+\frac{1}{288\omega_0^4}E_4(\tau_0).
\end{align*}
In a similar way, after some calculations, we get
\[
c_2(\tau_0,\omega_0) = -\frac{1}{864\omega_0^6} E_6(\tau_0) - 12 c_0(\tau_0,\omega_0) c_1(\tau_0,\omega_0) - 16 c_0(\tau_0,\omega_0)^3.
\]
To summarize, we obtain 
\begin{subequations}\label{eq:rationalityViaEisenstein}
\begin{equation} 
E_2^*(\tau_0)=-24c_0(\tau_0,\omega_0)\omega_0^2,
\end{equation}
\begin{equation} 
E_4(\tau_0)=288\left(c_1(\tau_0,\omega_0) + 2c_0(\tau_0,\omega_0)^2\right)\omega_0^4,
\end{equation}
\begin{equation}
E_6(\tau_0) =-864\left(c_2(\tau_0,\omega_0)+12c_0(\tau_0,\omega_0)c_1(\tau_0,\omega_0) + 16 c_0(\tau_0,\omega_0)^3\right)\omega_0^6.
\end{equation}
\end{subequations}
Equivalently, we have 
\begin{subequations}
\begin{equation} 
E_2^*(\tau_0)=-24c_0(\tau_0,\omega_0)\omega_0^2,
\end{equation}
\begin{equation} 
\p E_2^*(\tau_0)=-24c_1(\tau_0,\omega_0)\omega_0^4,
\end{equation}
\begin{equation} 
\p^2E_2^*(\tau_0)=-24c_2(\tau_0,\omega_0)\omega_0^6.
\end{equation}
\end{subequations}
This proves the theorem.
\end{proof}
\subsection{Examples}\label{examples}
\begin{proposition}[cf.\ Lemma~3.2 in \cite{Ma}] \label{proposition: real E_k}
The equation
\begin{equation}\label{eq:zeros}
E_2^*(\tau)= 0
\end{equation}
holds if and only if $\tau\in {\rm SL}(2,\ZZ)\sqrt{-1}$ or $\tau\in{\rm SL}(2,\ZZ)\rho$ where $\rho:=\exp \left(\frac{2\pi \sqrt{-1}}{3} \right)$. 
\end{proposition}

The values of the Eisenstein series at $\tau=\sqrt{-1}$ are
\begin{equation}\label{eq:i}
E_2(\sqrt{-1}) = \frac{3}{\pi}, \ E_4(\sqrt{-1}) = 3 \frac{\Gamma\left(\frac{1}{4}\right)^8}{64 \pi^6}, \ E_6(\sqrt{-1}) = 0.
\end{equation}
If 
\begin{equation*}
\omega_0\in\QQ \frac{\Gamma\left(\frac{1}{4}\right)^2}{4\pi^{\frac{3}{2}}}
\end{equation*}
then $c_0(\sqrt{-1},\omega_0)=c_2(\sqrt{-1},\omega_0)=0$ and $c_1(\sqrt{-1},\omega_0)\in\QQ$.
The values of the Eisenstein series at $\tau=\rho$ are
\begin{equation}\label{eq:rho}
E_2(\rho) = \frac{2\sqrt{3}}{\pi}, \ E_4(\rho) = 0, \ E_6(\rho) = \frac{27}{2} \frac{\Gamma\left(\frac{1}{3}\right)^{18}}{2^8\pi^{12}}.
\end{equation}
If 
\begin{equation*}
\omega_0\in\QQ \frac{\Gamma\left(\frac{1}{3}\right)^3}{4\pi^2}
\end{equation*}
then $c_0(\rho,\omega_0)=c_1(\rho,\omega_0)=0$ and $c_2(\rho,\omega_0)\in\QQ$.

\section{${\rm SL}$-action on the set of Frobenius manifolds $M^{(\tau_0,\omega_0)}$}\label{sec:Symmetries}
Let $A:=
\begin{pmatrix}
a &b\\
c & d
\end{pmatrix}$
be an element of ${\rm SL}(2,\RR)$. 
The correspondence
\begin{equation*}
\tau_0\mapsto \tau_1:=\frac{a\tau_0+b}{c\tau_0+d},\quad \omega_0\mapsto\omega_1:=(c\tau_0+d)\omega_0
\end{equation*}
defines a ${\rm SL}(2,\RR)$-action on the set 
$\left\{(\tau_0,\omega_0)~\vert~\tau_0\in\HH,\ \omega_0\in\CC\setminus\{0\}\right\}$.
This is exactly the ${\rm SL}(2,\RR)$-action induced by \eqref{eq:GL-action}
since 
\[
A
  \begin{pmatrix}
  \dfrac{\bar{\tau}_0}{4\pi\omega_0{\rm Im}(\tau_0)} & \omega_0\tau_0\\
  \dfrac{1}{4\pi\omega_0{\rm Im}(\tau_0)} & \omega_0
  \end{pmatrix}
=\begin{pmatrix}
  \dfrac{(a \bar \tau_0 + b)}{4\pi\omega_0{\rm Im}(\tau_0)} & (a\tau_0 + b) \omega_0\\
  \dfrac{(c \bar \tau_0 + d)}{4\pi\omega_0{\rm Im}(\tau_0)} & (c\tau_0+d)\omega_0
  \end{pmatrix}
=
 \begin{pmatrix}
  \dfrac{\bar{\tau}_1}{4\pi\omega_1{\rm Im}(\tau_1)} & \omega_1\tau_1\\
  \dfrac{1}{4\pi\omega_1{\rm Im}(\tau_1)} & \omega_1
  \end{pmatrix}.
\]
\subsection{${\rm SL(2,\ZZ)}$-action}
The equations \eqref{eq:rationalityViaEisenstein} yield the following.
\begin{proposition}\label{proposition: SL2Z action}
Let $\tau_0,\tau_1\in\HH$ and $\omega_0,\omega_1\in\CC\backslash\{0\}$. The following are equivalent$:$
\begin{enumerate}
\item
There is an isomorphism of Frobenius manifolds $M^{(\tau_0,\omega_0)}\cong M^{(\tau_1,\omega_1)}$.
\item
The equality $f^{(\tau_0,\omega_0)}(t)=f^{(\tau_1,\omega_1)}(t)$ holds.
\item
There exists an element 
$\begin{pmatrix}
    a & b\\
    c & d
\end{pmatrix} \in {\rm SL}(2, \ZZ)
$ such that 
\begin{equation*}
\tau_1=\frac{a\tau_0+b}{c\tau_0+d},\quad \omega_1^k=(c\tau_0+d)^k\omega_0^k,
\end{equation*}
where $k=4$ if $\tau_0\in {\rm SL}(2, \ZZ)\sqrt{-1}$, $k=6$ if $\tau_0\in {\rm SL}(2, \ZZ)\rho$ and $k=2$ otherwise. 
\end{enumerate}
\end{proposition}
\begin{proof}
It is almost clear that condition (i) is equivalent to (ii).  
By the equations \eqref{eq:rationalityViaEisenstein}, condition (ii) is equivalent to the equations
\begin{equation}\label{42}
\frac{E_2^*(\tau_0)}{\omega_0^2}=\frac{E_2^*(\tau_1)}{\omega_1^2},\quad 
\frac{E_4(\tau_0)}{\omega_0^4}=\frac{E_4(\tau_1)}{\omega_1^4},\quad 
\frac{E_6(\tau_0)}{\omega_0^6}=\frac{E_6(\tau_1)}{\omega_1^6}.
\end{equation}
The equations \eqref{42} imply that $j(\tau_0)=j(\tau_1)$ and hence
\begin{equation}
\tau_1=\frac{a\tau_0+b}{c\tau_0+d},\quad \text{for some} \quad \begin{pmatrix}
    a & b\\
    c & d
\end{pmatrix} \in {\rm SL}(2, \ZZ).
\end{equation}
Therefore we obtain
\[
\frac{E_2^*(\tau_0)}{\omega_0^2}=\frac{(c\tau_0+d)^2E_2^*(\tau_0)}{\omega_1^2},\  
\frac{E_4(\tau_0)}{\omega_0^4}=\frac{(c\tau_0+d)^4E_4(\tau_0)}{\omega_1^4},\ 
\frac{E_6(\tau_0)}{\omega_0^6}=\frac{(c\tau_0+d)^6E_6(\tau_0)}{\omega_1^6},
\]
which implies, by the use of \eqref{eq:zeros}, \eqref{eq:i} and \eqref{eq:rho},  $\omega_1^k=(c\tau_0+d)^k\omega_0^k$ 
where $k=4$ if $\tau_0\in {\rm SL}(2, \ZZ)\sqrt{-1}$, $k=6$ if $\tau_0\in {\rm SL}(2, \ZZ)\rho$ and $k=2$ otherwise. 
It is easy to show that the condition (iii) yields the equations \eqref{42}. The proposition is proved.
\end{proof}
\subsection{${\rm SL(2,\QQ)}$-action and complex multiplication}
\begin{definition}
  An elliptic curve $\E_{\tau}$ is said to have complex multiplication if its modulus $\tau$ is imaginary quadratic. Namely $\tau \in \mathbb Q(\sqrt{-D})$ for a positive integer $D$.
\end{definition}

A profound result of the theory of elliptic curves is that elliptic curves over $\QQ$ with complex multiplication are easily classified:

\begin{theorem}[cf.\ Paragraph~II.2 in \cite{Si}]\label{prop:13ellipticCurves}
  Up to isomorphism there are only 13 elliptic curves defined over $\QQ$ that have complex multiplication.
\end{theorem}

We give the list of the Weierstrass models of these elliptic curves in the Appendix. 
\begin{corollary}\label{cor:13EllipticModulus}
  The modulus $\tau_0$ of the elliptic curve $\E_{\tau_0}$ with complex multiplication defined over $\QQ$ is in the $\rm SL(2, \CC)$ orbit of one of:
  $$
    \sqrt{-D}, \quad D \in \{ 1, 2, 3, 4 ,7\},
  $$
  or
  $$
    \frac{-1 + \sqrt{-D}}{2}, \quad D \in \{3, 7, 11, 19, 27, 43, 67, 163\}.
  $$
\end{corollary}

Imaginary quadratic $\tau_0 \in \CC$ are amazing from the point of view of the theory of modular forms too:

\begin{proposition}[cf.\ Theorem~A1 in \cite{Ma}]\label{prop:E2E4E6}
  Let $\tau \in \CC$ be imaginary quadratic and ${\tau \not\in {\rm SL}(2,\ZZ) \sqrt{-1}}$. Then we have:
  $$
    \frac{E_2^*(\tau)E_4(\tau)}{E_6(\tau)} \in \QQ (j(\tau)),
  $$
  where $j(\tau)$ is the value of the $j$-invariant of the elliptic curve $\E_{\tau}$.
\end{proposition}

\begin{definition}\label{defn:symmetry}
Let $\tau_0 \in \HH$, $\omega_0 \in \CC \backslash \{0\}$.
\begin{enumerate}
\item
The Frobenius manifold $M^{(\tau_0,\omega_0)}$ is said to {\em have a symmetry} if there exists an element $A\in {\rm SL}(2,\RR)\setminus \{1\}$ such that 
\[
\left(f^{(\tau_0, \omega_0)}\right)^A(t) = f^{(\tau_0, \omega_0)}(t).
\]
\item
The Frobenius manifold $M^{(\tau_0,\omega_0)}$ is said to {\em have a weak symmetry} if there exists an element $A\in {\rm SL}(2,\RR)\setminus \{1,-1\}$ such that 
\[
\left(f^{(\tau_0, \omega_0)}\right)^A(t) = f^{(\tau_0, \omega_0^\prime)}(t)\quad \text{for some }\omega_0^\prime \in \CC \backslash \{0\}.
\]
\end{enumerate}
\end{definition}

\begin{remark}It is important to note that weak symmetry is not a symmetry of the Frobenius manifold unless $\omega_0 = \omega_0^\prime$, because the corresponding $A$-action relates different points in the space of all Frobenius manifolds of rank three .
\end{remark}

\begin{theorem}\label{theorem2} 
  Let $\tau_0 \in \HH$ and $\omega_0 \in \CC \backslash \{0\}$.
  \begin{enumerate}
    \item The Frobenius manifold $M^{(\tau_0, \omega_0)}$ has a symmetry if and only if $\tau_0$ is in the $\rm SL(2, \ZZ)$ orbit of $\sqrt{-1}$ or $\rho$.
    \item The Frobenius manifold $M^{(\tau_0, \omega_0)}$ defined over $\QQ$ has a weak symmetry if and only if $\tau_0$ is from the list given in Corollary~\ref{cor:13EllipticModulus}.
  \end{enumerate}
\end{theorem}
\begin{proof}
  From Proposition~\ref{proposition: SL2Z action} $M^{(\tau_0, \omega_0)}$ has a symmetry if and only if $\frac{a\tau_0+b}{c\tau_0+d}= \tau_0$ and
  $$
    \omega_0^4 = (c \tau_0 + d)^4 \omega_0^4 \quad \text{for} \quad \tau_0 \in \rm SL(2 ,\ZZ) \sqrt{-1},
  $$
  or
  $$
    \omega_0^6 = (c \tau_0 + d)^6 \omega_0^6 \quad \text{for} \quad \tau_0 \in \rm SL(2 ,\ZZ) \rho,
  $$
  or otherwise
  $$
    \omega_0^2 = (c \tau_0 + d)^2 \omega_0^2.
  $$
  The last equation is satisfied if and only if $(c \tau_0 + d)^2 = 1$. It has no solutions for $\tau_0 \in \HH$ and $c,d \in \ZZ$. It is an easy exercise to show that there is a suitable $A \in \rm SL(2, \ZZ)$ solving the first two equations. This proves (i).

  Let $M^{(\tau_0, \omega_0)}$ be defined over $\QQ$ and have a weak symmetry. By Theorem~\ref{theorem1} the elliptic curve $\E_{\tau_0}$ is defined over $\QQ$.

  Due to Proposition~\ref{proposition: SL2Z action} we have $\frac{a\tau_0+b}{c\tau_0+d} = \tau_0$. It is an easy exercise to show that $\tau_0$ satisfies a quadratic equation with negative discriminant. Hence the elliptic curve $\E_{\tau_0}$ has complex multiplication. From Proposition~\ref{prop:13ellipticCurves} we know that there are only 13 such $\tau_0$ up to the $\rm SL(2, \ZZ)$-action.
  Hence $\tau_0$ is from the given list.

  Assume that $\tau_0$ is the modulus of one of the elliptic curves from this list. From the rationality assumption on the elliptic curve $\E_{\tau_0}$ we have $j(\tau_0) \in \QQ$. 
  The case of $\tau_0 = {\rm SL}(2, \ZZ) \sqrt{-1}$ was treated in Example~\ref{examples} and we can apply Proposition~\ref{prop:E2E4E6}. Its statement reads:
  $$
    \frac{E_2^*(\tau_0)E_4(\tau_0)}{E_6(\tau_0)} \in \QQ.
  $$

  At the same time, since the elliptic curve is defined over $\QQ$, there exists $a \in \CC \backslash \{0\}$ such that:
  $$
    a^2  g_2(\tau_0) \in \QQ, \quad a^3 g_3(\tau_0) \in \QQ.
  $$

  From the equations~\eqref{eq:g to Eisentein} we have:
  $$
    a^2 \pi^4 E_4(\tau_0) = a^2 g_2(\tau_0) \frac{3}{4} \in \QQ, \quad a^3 \pi^6 E_6(\tau_0) = a^3 g_3(\tau_0) \frac{27}{8 } \in \QQ.
  $$

  We conclude:
  $$
    a \pi^2 E_2^*(\tau_0) \in \QQ.
  $$
  Summing up:
  $$
    E_2^*(\tau_0) \in \QQ (a \pi^2)^{-1}, \quad E_4(\tau_0) \in \QQ (a\pi^2)^{-2}, \quad E_6(\tau_0) \in \QQ (a \pi^2)^{-3}.
  $$
  Taking $\omega_0^2 := (a \pi^2)^{-1}$ we get $M^{(\tau_0,\omega_0)}$ defined over $\QQ$ because of Theorem~\ref{theorem1}.
\end{proof}

  Note that $\E_{\tau_0}$ for $\tau_0 \in \rm SL(2, \ZZ) \sqrt{-1}$ and $\tau_0 \in \rm SL(2, \ZZ) \rho$ are only elliptic curves with non-trivial automorphisms.
\begin{remark}
  We can rephrase Theorem \ref{theorem2} (i) above as: a Frobenius manifold $M^{(\tau_0, \omega_0)}$ has a symmetry if and only if $\E_{\tau_0}$ has non-trivial automorphisms.
\end{remark}

\section{Appendix}\label{Appendix} 

  13 elliptic curves over $\mathbb Q$ with complex multiplication.
  
  \begin{tabular}{ c | c | c | c}\label{table:13ellipticCurves}
    Modulus $\tau$ & Weierstrass equation & $j$-invariant & $\Delta_E$
    \\
    \hline
    $(-1 + \sqrt{-3})/2$ & $y^2 = 4x^3 + 1$ & $0$ & $3^3$
    \\
    $\sqrt{-3}$ & $y^2 = 4x^3 - 60 x + 88$ & $2^43^35^3$ & $2^83^3$
    \\
    $(-1 + 3\sqrt{-3})/2$ & $y^2 = 4x^3 - 120x + 253$ & $-2^{15} 3 5^3$ & $3^5$
    \\
    $\sqrt{-1}$ & $y^2 = 4x^3 + 4x$ & $2^63^3$ & $2^5$
    \\
    $2\sqrt{-1}$ & $ y^2 = 4 x^3 - 44 x + 64$ & $2^3 3^3 11^3$ & $2^9$
    \\
    $(-1 + \sqrt{-7})/2$ & $ y^2 = 4 x^3 - \frac{35}{4} x - \frac{49}{8}$ & $-3^3 5^3$ & $7^3$
    \\
    $\sqrt{-7}$ & $y^2 = 4 x^3 - 2380 x + 22344$ & $3^3 5^3 17^3$ & $2^{12} 7^3$
    \\
    $\sqrt{-2}$ & $y^2 = 4 x^3 - 120 x + 224$ & $2^6 5^3$ & $2^9$
    \\
    $(-1 + \sqrt{-11})/2$ & $y^2 = 4 x^3 - \frac{88}{3}x - \frac{847}{27} $ & $-2^{15}$ & $11^3$
    \\
    $(-1 + \sqrt{-19})/2$ & $y^2 = 4 x^3 - 152 x + 361$ & $-2^{15}3^3$ & $19^3$
    \\
    $(-1 + \sqrt{-43})/2$ & $y^2 = 4 x^3 - 3440 x + 38829$ & $-2^{18}3^3 5^3$ & $43^3$
    \\
    $(-1 + \sqrt{-67})/2$ & $y^2 = 4 x^3 - 29480 x + 974113$ & $-2^{15} 3^3 5^3 11^3$ & $67^3$
    \\
    $(-1 + \sqrt{-163})/2$ & $y^2 = 4 x^3 - 8697680 x + 4936546769$ & $-2^{18} 3^3 5^3 23^3 29^3$ & $163^3$
  \end{tabular}
  
  Introduce the notation:
  $$
    \psi(\tau) := \frac{3 E_2^*(\tau) E_4(\tau) }{2E_6(\tau)}.
  $$
  The values of this function where computed in \cite{Ma}. Using $g_2,g_3$ given by the Weierstrass forms of the previous table we compute $c_0(\tau,\omega)$, $c_1(\tau,\omega)$, $c_2(\tau,\omega)$ for some choice of $\omega$:

  \begin{tabular}{ c | c| c | c | c}\label{table:13ellipticCurves}
    Modulus $\tau$ & $\psi(\tau)$ & $c_0(\tau,\omega)$ & $c_1(\tau,\omega)$ & $c_2(\tau,\omega)$
    \\
    \hline
    $(-1 + \sqrt{-3})/2$ & $0$ & $0$ & $0$ & $-1/256$
    \\
    $\sqrt{-3}$ & $15 / 22$ & $1/16$ & $- 21/128 $ & $- 115/512 $
    \\
    $(-1 + 3\sqrt{-3})/2$ & $240/253$ & $ 1/8 $ & $- 11/32$ & $- 129/256$
    \\
    $\sqrt{-1}$ & $\infty$ & $0$ & $1/96$ & $0$
    \\
    $2\sqrt{-1}$ & $11/14$ & $ 1/16$ & $- 47/384$ & $- 67/512$
    \\
    $(-1 + \sqrt{-7})/2$ & $5/14$ & $-1/64$ & $- 143/6144$ & $ 643/32768 $
    \\
    $\sqrt{-7}$ & $255 / 266$ & $ 9/16$ & $- 2623/384 $ & $- 22539/512$
    \\
    $\sqrt{-2}$ & $15/28$ & $1/48$ & $-41/1152$ & $- 109/4608$
    \\
    $(-1 + \sqrt{-11})/2$ & $ 48/77$ & $ 1/24$ & $- 23/288$ & $- 193/2304$
    \\
    $(-1 + \sqrt{-19})/2$ & $ 16/19$ & $1/8$ & $-41/96$ & $-205/256$
    \\
    $(-1 + \sqrt{-43})/2$ & $320/301$ & $3/4$ & $-121/12$ & $-17325/256$
    \\
    $(-1 + \sqrt{-67})/2$ & $ 16720/14539$ & $19/8$ & $-8453/96$ & $-386557/256$
    \\
    $(-1 + \sqrt{-163})/2$ & $38632640/30285563$ & $181/4$ & $-80236/3$ & $-1598234897/256$
  \end{tabular}

\end{document}